\documentclass[graybox]{svmult}

\usepackage{amsfonts}
\usepackage{amsmath}
\usepackage{mathptmx}       
\usepackage{helvet}         
\usepackage{courier}        
\usepackage{type1cm}        

\usepackage{makeidx}         
\usepackage{graphicx}        
\usepackage{multicol}        
\usepackage[bottom]{footmisc}
\usepackage[hidelinks]{hyperref}        
\usepackage{apacite} 
\usepackage[hidelinks]{hyperref}

\usepackage[spanish]{babel}
\let\oldbibliography\thebibliography
\renewcommand{\thebibliography}[1]{%
  \oldbibliography{#1}%
  \setlength{\itemsep}{12pt}%
}

\newtheorem{prop}{Proposici\'on}

\newtheorem{ej}{Ejemplo}

\newtheorem{defi}{Definici\'on}

\newtheorem{observ}{Observaci\'on}

\usepackage[all]{xy}
\usepackage{color}
\usepackage[utf8]{inputenc}
\usepackage{float}

\makeindex 

\begin{document}

\title*{Condiciones para obtener factorizaciones únicas en la categoría
$\text{\textbf{Rel}}(D^{\#})$  \\
\bigskip Conditions for uniqueness of factorizations on the $\text{\textbf{Rel}}(D^{\#})$ category }
\titlerunning{Factorizaciones en $\text{\textbf{Rel}}(D^{\#})$ }
\author{David Fernando Méndez Oyuela}
\institute{ David Fernando Méndez Oyuela, M.S., M.Ed. \at Profesor Titular II, Departamento de Matemática Pura, Universidad Nacional Autónoma de Honduras (UNAH), Tegucigalpa, Honduras, \email{david.mendez@unah.edu.hn}}
%
%
\maketitle
\abstract{En este trabajo se analiza un posible vínculo entre la Teoría de Categorías y la Teoría de Factorizaciones Generalizadas desarrolladas por Anderson y Frazier. Específicamente bajo el contexto de trabajos previos, donde se analizan composiciones de relaciones y su vínculo con las $\tau$-factorizaciones. \\
\textbf{Palabras Claves}  categorías, dominios integrales, $\tau$-factorizaciones\\  
\newline
\textbf{Abstract} This paper analyzes a possible link between Category Theory and Generalized Factorization Theory developed by Anderson and Frazier. Specifically in the context of what has been worked on in previous works, where compositions of relations and their link with $\tau$-factorizations are analyzed. 
}

\section{Introducción}

Inspirados en los trabajos de \shortciteA{mcadam} y \shortciteA{zafru}, \shortciteA{frazier} desarrollaron el concepto
de $\tau$-factorzación, el mismo fue estudiado en forma más amplia
por, entre otros, \shortciteA{ortiz} y \shortciteA{juett}, más adelante
Ortiz desarrolló varios proyectos de investigación sobre el tema,
para más detalles, ver por ejemplo \shortciteA{vargas}, \shortciteA{serna}, \shortciteA{molina},
\shortciteA{barrios} y \shortciteA{calderon}. La maquinaria de la teoría de categorías ofrece una forma
de determinar qué condiciones se deben dar para que dada una relación
$\tau$ en un conjunto $D^{\#}$ (más adelante se especificará sobre
esta notación) se pueden conseguir dos relaciones $\tau_{1}$ y $\tau_{2}$
tales que $\tau=\tau_{1}\circ\tau_{2}$, cuestión que brinda información
útil para los objetivos del estudio de \shortciteA{mendez}. Se presentan en la siguiente sección los conceptos necesarios para comprender
este proceso, los mismos son basados en \shortciteA{castellini}
y \shortciteA{adamek}.

\section{Conceptos básicos}
\begin{defi} 
\normalfont
Una categoría es una cuarteta $\mathcal{A}=\left\{ \mathcal{O},\,\text{hom},\,id,\,\circ\right\} $
que consiste de:

\begin{enumerate}
\item una clase $\mathcal{O}$, cuyos miembros se llaman $\mathcal{A}$-objetos,
\item para cada par $(A,B)$ de $\mathcal{A}$-objetos, un conjunto $\text{hom}(A,B)$
, cuyos miembros se llaman $\mathcal{A}$-morfismos de $A$ a $B$.
Un $\mathcal{A}$-morfismo $f\in\text{hom}(A,B)$ se suele denotar
como $f:A\longrightarrow B$ ó $A\overset{f}{\longrightarrow}B$,
\item para cada $\mathcal{A}$-objeto $A$, un morfismo $A\overset{id_{A}}{\longrightarrow}A$,
llamado la $A$-identidad en $A$,
\item una ley de composición, que asocia a cada $\mathcal{A}$-morfismo
$A\overset{f}{\longrightarrow}B$ y cada $\mathcal{A}$-morfismo $B\overset{g}{\longrightarrow}C$
otro $\mathcal{A}$-morfismo $A\overset{g\circ f}{\longrightarrow}C$
(llamado la composición de $f$ con $g$) con las siguientes condiciones:
\begin{enumerate}
\item la composición es asociativa: dados los morfismos $A\overset{f}{\longrightarrow}B$,
$B\overset{g}{\longrightarrow}C$ y $C\overset{h}{\longrightarrow}D$,
la ecuación $h\circ(g\circ f)=(h\circ g)\circ f$ se satisface,
\item $A$-identidades actúan como identidades respecto a la composición:
dado el $\mathcal{A}$-morfismo $A\overset{f}{\longrightarrow}B$,
se tiene que $id_{B}\circ f=f$ y $f\circ id_{A}=f$,
\item los conjuntos $\text{hom}(A,B)$ son mutuamente exclusivos.
\end{enumerate}
Por simplicidad, cuando no hay ambigüedad sobre la categoría en cuestión,
a los $\mathcal{A}$-objetos y $\mathcal{A}$-morfismos se les llama
objetos y morfismos respectivamente. También una categoría $\mathcal{A}=\left\{ \mathcal{O},\,\text{hom},\,id,\,\circ\right\} $
se denota simplemente como $\mathcal{A}$.

\end{enumerate}
\end{defi} 

\begin{observ} 
\normalfont
Dada una categoría $\mathcal{A}=(\mathcal{O},\text{hom},id,\circ)$,
entonces se tienen las siguientes observaciones:

\begin{enumerate}
\item Aunque la notación que se usa para los $\mathcal{A}$-morfismos es la misma que se usa para funciones, no necesariamente un $\mathcal{A}$-morfismo
es una función (ver Ejemplos).
\item La clase $\mathcal{O}$ de $\mathcal{A}$-objetos es denotada por
$Ob(\mathcal{A})$. Obsérvese la distinción al mencionar conjuntos
y clases, esto se hace para evitar dificultades técnicas en términos
de paradojas que se pueden presentar si ambos se consideran como un
mismo ente. Se denomina clase a una colección de conjuntos y conglomerado
a una colección de clases. La palabra colección se usa como término
primitivo para hablar en general de la noción de conjuntos.
\item La clase de todos los $\mathcal{A}$-morfismos (denotada por $Mor(\mathcal{A})$)
está definida como la unión de todos los conjuntos $\text{hom}(A,B)$
en $\mathcal{A}$.
\item Si $A\overset{f}{\longrightarrow}B$ es un $\mathcal{A}$-morfismo,
a $A$ se le llama el dominio de $f$ (denotado por $Dom(f)$) y a
$B$ el codominio de $f$ (denotado por $Codom(f)$). Obsérvese que
la condición (c) garantiza que cada $\mathcal{A}$-morfismo tiene
un único dominio y un único codominio. Sin embargo, esta condición
está dada solamente por razones técnicas pues cuando las demás condiciones
se satisfacen, es "fácil'' forzar la condición (c) reemplazando
cada morfismo $f\in\text{hom}(A,B)$ por una tripleta $(A,f,B)$,
por esta razón, cuando se desea comprobar que un objeto es una categoría,
se hace caso omiso de la condición (c).
\item La composición, $\circ$, es una operación binaria parcial en la clase
$Mor(\mathcal{A})$. Es decir, para cada par $(f,g)$ de morfismos,
$f\circ g$ está definida si y solo si el dominio de $f$ coincide
con el codominio de $g$.
\item Si más de una categoría se está involucrando, es conveniente introducir
subíndices, por ejemplo: $\text{hom}_{\mathcal{A}}(A,B)$.\\
Se considera ahora una serie de ejemplos para comprender la definición
del concepto de categoría. 
\end{enumerate}
\end{observ} 
\begin{ej}
\normalfont
\begin{enumerate}
\item La categoría \textbf{Set}, cuya clase de objetos es la clase de
conjuntos, $\text{hom}(A,B)$ es el conjunto de todas las funciones
de $A$ a $B$, $id_{A}$ es la función identidad en $A$ y $\circ$
es la composición usual de funciones.
\item La categoría \textbf{Vec}, cuya clase de objetos son todos los
espacios vectoriales, $\text{hom}(V,W)$ es el conjunto de todas las
transformaciones lineales de $V$ a $W$, $id_{V}$ es la transformación
lineal identidad y $\circ$ es la composición de transformaciones
lineales.
\item La categoría \textbf{Grp}, cuyos objetos son todos los grupos,
los morfismos son los homomorfismos entre grupos, $id_{G}$ es el
homomorfismo identidad y $\circ$ es la composición de homomorfismos.
\item La categoría \textbf{Top}, cuyos objetos son todos los espacios
topológicos, los morfismo son las funciones continuas entre espacios
topológicos, la identidad es la función continua identidad(entre el
mismo espacio topológico) y $\circ$ la composición de funciones continuas.
\item La categoría \textbf{Rel}, cuyos objetos son los conjuntos y los
morfismos son relaciones binarias entre conjuntos, la identidad es
la relación identidad y $\circ$ es la composición de relaciones.
El lector notará que ésta es la categoría de interés en este estudio.
\end{enumerate}
\end{ej} 
\hspace{1cm}Debido a que en muchas categorías los morfismos son funciones,
se adopta la notación que usualmente se utiliza para éstas. Al morfismo
$h=g\circ f$ a veces se denota como $A\overset{f}{\longrightarrow}B\overset{g}{\longrightarrow}C$
o diciendo que el diagrama triangular 
\[
\xymatrix{A\ar[r]^{f}\ar[rd]_{h} & B\ar[d]^{g}\\
 & C
}
\]
es conmutativo. Similarmente, al decir que el diagrama cuadrado 
\[
\xymatrix{A\ar[r]^{f}\ar[d]_{h} & B\ar[d]^{g}\\
C\ar[r]_{k} & D
}
\]
conmuta, esto implica que $g\circ f=k\circ h$.\\
Los siguientes conceptos relacionados a morfirmos son necesarios para
comprender la ``estructura de factorización'' que permitirá decidir
cuando un morfismo $f$ se puede factorizar como $f=g\circ h$, con
particular interés en el caso de la categoría \textbf{Rel}.\bigskip
\begin{defi}
\normalfont
Sea $\mathcal{X}$ una categoría. 
\begin{enumerate}
\item Un morfismo $X\overset{f}{\longrightarrow}Y$ en $\mathcal{X}$ es
llamado un isomorfismo, si existe un morfismo $Y\overset{g}{\longrightarrow}X$
tal que $g\circ f=id_{X}$ y $f\circ g=id_{Y}$.
\item Un morfismo $M\overset{m}{\longrightarrow}X$ en $\mathcal{X}$ es
llamado monomorfismo, si para todo $f,\,g\,:\,Y\longrightarrow M$
morfismos en $\mathcal{X}$, tales que si $m\circ f=m\circ g$ , entonces
$f=g$. También se dice que el par $(M,m)$ (o simplemente $m$) es
un subobjeto de $X$.
\item Un morfismo $X\overset{e}{\longrightarrow}E$ en $\mathcal{X}$ es
llamado epimorfismo, si para todo $f,\,g\,:\,E\longrightarrow Y$
morfismos en $\mathcal{X}$ tales que $f\circ e=g\circ e$, entonces
$f=g$.
\item Un morfismo $M\overset{m}{\longrightarrow}X$ en $\mathcal{X}$ es
llamado sección, si existe un morfismo $X\overset{f}{\longrightarrow}M$
tal que $f\circ m=id_{M}$.
\item Un morfismo $X\overset{e}{\longrightarrow}E$ en $\mathcal{X}$ es
llamado retracción, si existe un morfismo $E\overset{g}{\longrightarrow}X$
tal que $e\circ g=id_{E}$.
\item Una familia de morfismos con dominio común $\left(X\overset{f_{i}}{\longrightarrow}Y_{i}\right)_{i\in I}$,
indexada por una clase $I$, es llamada una fuente (source). Similarmente
para morfismos de codominio común se define un sumidero (sink).
\item Una fuente $\left(X\overset{f_{i}}{\longrightarrow}Y_{i}\right)_{i\in I}$
es llamada mono-fuente (monosource) si por cada par de morfismos $h,\,k\,:\,Z\longrightarrow X$,
$f_{i}\circ h=f_{i}\circ k$, para cada $i\in I$ implica que $h=k$.
Similarmente, para sumideros se define el concepto de epi-sumidero
(episink). 
\end{enumerate}
\end{defi}
Mediante el uso de estos conceptos se puede introducir ahora la teoría
de estructuras de factorización para sumideros. Se asume que todos
los objetos y morfismos pertenecen a una categoría arbitraria pero
fija $\mathcal{X}$.\bigskip
\begin{defi}
\normalfont
Sea $\mathbf{E}$ un conglomerado de sumideros
y sea $\mathcal{M}$ una clase de morfismos. Se dice que $(\mathbf{E},\mathcal{M})$
es una estructura de factorización (para sumideros) en la categoría
$\mathcal{X}$ y que $\mathcal{X}$ es una $(\mathbf{E},\mathcal{M})$-categoría
(para sumideros) si:

\begin{enumerate}
\item tanto $\mathbf{E}$ como $\mathcal{M}$ son cerrados bajo composiciones
con isomorfismos, en particular, esto para $\mathbf{E}$ significa
que si $\left(X_{i}\overset{e_{i}}{\longrightarrow}Y\right)_{i\in I}$
es un sumidero en $\mathbf{E}$ y $Y\overset{h}{\longrightarrow}Z$
es un isomorfismo, entonces el sumidero $\left(X_{i}\stackrel{h\circ e_{i}}{\longrightarrow}Z\right)_{i\in I}$
está en $\mathbf{E}$;
\item $\mathcal{X}$ tiene $(\mathbf{E},\mathcal{M})$-factorizaciones (de
sumideros); es decir, cada sumidero $\mathbf{s}$ en $\mathcal{X}$
tiene una factorización $\mathbf{s}=m\circ\mathbf{e}$ donde $\mathbf{e}\in\mathbf{E}$
y $m\in\mathcal{M}$; 
\item $\mathcal{X}$ tiene una única $(\mathbf{E},\mathcal{M})$-propiedad
de diagonalización; es decir, si $Y\overset{s}{\longrightarrow}Z$
y $M\overset{m}{\longrightarrow}Z$ son $\mathcal{X}$-morfismos con
$m\in\mathcal{M}$ y $\mathbf{e}=\left(X_{i}\overset{e_{i}}{\longrightarrow}Y\right)_{i\in I}$
y $\mathbf{r}=\left(X_{i}\overset{r_{i}}{\longrightarrow}M\right)_{i\in I}$
son sumideros en $\mathcal{X}$ con $\mathbf{e}\in\mathbf{E}$, tales
que $m\circ\mathbf{r}=s\circ\mathbf{e}$ , entonces existe un único
morfismo diagonal $Y\overset{d}{\longrightarrow}M$ tal que para cada
$i\in I$ el siguiente diagrama conmuta: 
\[
\xymatrix{X_{i}\ar[r]^{e_{i}}\ar[d]_{r_{i}} & Y\ar[d]^{s}\ar[dl]_{d}\\
M\ar[r]_{m} & Z
}
\]
Cabe destacar que cualquier $(\mathbf{E},\mathcal{M})$-categoría
para sumideros es también una $(\mathcal{E},\mathcal{M})$-categoría
para morfismos individules, donde $\mathcal{E}$ consiste de todos
los morfismos (que se pueden ver como sumideros con un solo morfismo)
que pertenecen a $\mathbf{E}$. En la Sección \ref{sec:Estructura-de-factorizaci=0000F3n}
se hará el estudio sobre esta estructura de factorización aplicada
particularmete a la categoría \textbf{Rel}. 
\end{enumerate}
\end{defi} 

\section{Estructura de factorización en la categoría $\text{\textbf{Rel}}(D^{\#})$\label{sec:Estructura-de-factorizaci=0000F3n}}

\hspace{1cm}Como se introdujo en la sección anterior, dada la categoría
\textbf{Rel}, formada por conjuntos y relaciones entre ellos, se puede
construir una $(\mathbf{E},\mathcal{M})$-estructura de factorización
para conglomerados de sumideros, donde $\mathbf{E}$ es un conglomerado
de sumideros y $\mathcal{M}$ una clase de morfismos, además en particular
también se tiene una $(\mathcal{E},\mathcal{M})$-estructura de factorización
donde $\mathcal{E}$ es el conjunto de sumideros singuletes(vistos
como conglomerados con un solo elemento). En \shortciteA{castellini}
(tomando únicamente los puntos de más interés para este estudio) se prueba
el siguiente resultado donde especifica qué propiedades se tienen
cuando una categoría tiene una $(\mathbf{E},\mathcal{M})$-estructura
de factorización.
\begin{prop}
\label{proposition:4.9}
Si $\mathcal{X}$ es una $(\mathbf{E},\mathcal{M})$-categoría(para
sumideros) entonces $\mathbf{E}$ y $\mathcal{M}$ tienen las siguientes
propiedades:

\begin{enumerate}
\item $\mathcal{M}$ consiste de monomorfismos y $\mathbf{E}$ contiene
todos los epi-sumideros extremales.(Un epimorfismo $X\overset{e}{\longrightarrow}E$
es extremal si cuando se factoriza como $e=m\circ f$, con $m$ monomorfismo,
entonces $m$ es isomorfismo).
\item $\mathcal{M}$ contiene todos los isomorfismos y es cerrado bajo composición.
\item $\mathbf{E}$ es cerrado bajo composición, en el sentido que si $\left(X_{i}\overset{e_{i}}{\longrightarrow}Y\right)_{i\in I}$
es un sumidero en $\mathbf{E}$ y el morfismo $Y\overset{f}{\longrightarrow}Z$(visto
como un sumidero singulete) pertenece a $\mathbf{E}$ entonces el
sumidero $\left(X_{i}\overset{f\circ e_{i}}{\longrightarrow}Z\right)_{i\in I}$
también está en $\mathbf{E}$.
\item Las $(\mathbf{E},\mathcal{M})$-factorizaciones son esencialmente
únicas, esto es, si $\left((e_{i})_{i\in I},\,m\right)$ y $\left((f_{i})_{i\in I},\,n\right)$
son dos $(\mathbf{E},\mathcal{M})$-factorizaciones para el mismo
sumidero, entonces existe un isomorfismo $h$ tal que para cada $i\in I$,
el diagrama
\[
\xymatrix{X_{i}\ar[r]^{e_{i}}\ar[d]_{f_{i}} & M\ar[d]^{m}\ar[dl]_{h}\\
N\ar[r]_{n} & X
}
\]
conmuta.
\item $\mathcal{M}\cap\mathbf{E}$ consiste de todos los isomorfismos. 
\item $\mathcal{M}$ es cerrado bajo primeros factores relativos a $\mathcal{M}$,
es decir, si $n\circ m\in\mathcal{M}$ y $n\in\mathcal{M}$, entonces
$m\in\mathcal{M}$, en consecuencia, si $(e,n)$ es la $(\mathbf{E},\mathcal{M})$-factorización
de $m\in\mathcal{M}$, entonces $e$ debe ser un isomorfismo.
\item Si $\mathbf{E}$ es un conglomerado de episumideros, entonces $\mathbf{E}$
es cerrado bajo segundos factores relativos a $\mathbf{E}$, es decir,
si $g\circ f\in\mathbf{E}$ y $f\in\mathbf{E}$ entonces $g\in\mathbf{E}$.
\end{enumerate}
\end{prop}

\hspace{1cm}Cabe destacar que en el caso de este estudio, se trabaja
con una subcategoría de \textbf{Rel}, que tiene por objetos a todos
los subconjuntos de $D^{\#}$, donde $D$ es un dominio integral arbitrario,
$D^{\#}$ es el conjunto de elementos distintos de cero y no invertibles
de $D$ y cuyos morfismos son todas las relaciones binarias entre
subconjuntos de $D^{\#}$, se denota esta subcategoría como $\text{\textbf{Rel}}(D^{\#})$.\\
Para poder utilizar el resultado anterior al caso particular de \textbf{Rel}$(D^{\#})$,
se necesita comprender todos los conceptos involucrados aplicados
a dicha categoría. Se procede a analizar entonces cada concepto. Se
caracterizan primero los conceptos de monomorfismo, epimorfismo e
isomorfismo en \textbf{Rel$(D^{\#})$}. Los siguientes conceptos y
resultados están basados en el trabajo de \shortciteA{fayn}.\bigskip
\begin{defi}
\normalfont
 Dado un morfismo $\tau\in\text{\textbf{Rel}}(D^{\#})$.

\begin{enumerate}
\item La imagen de $\tau$ se define como $Im\tau=\left\{ b\in B\,:\,\exists a\in A,\,\text{con }a\tau b\right\} $.
Obsérvese que $Im\tau\subseteq Codom\tau$.
\item La coimagen de $\tau$ se define como $Coim\tau=Im\text{\ensuremath{\tau}}^{-1}$.
\item $\tau$ es una correspondencia si $Coim\tau=Dom\tau$.
\item $\tau$ es una función parcial si para todo $x\in Dom\tau$, $\text{card}\tau\left[x\right]\leq1$.
(donde $\text{card}\tau\left[x\right]$ respresenta la cardinalidad
del conjunto $\tau\left[x\right]=\left\{ y\in Codom\tau\,:\,x\tau y\right\} $).
\item $\tau$ es sobreyectiva si $Im\tau=Codom\tau$. 
\item $\tau$ es inyectiva si para todo $x\in Codom\tau$, $\text{card}\tau^{-1}\left[x\right]\leq1$.
\item $\tau$ es biyectiva si es inyectiva y sobreyectiva. 
\item $\tau$ es una función si $\tau$ es una función parcial y una correspondencia. 
\end{enumerate}
\end{defi}
Considérese el siguiente ejemplo para comprender la definición: 
\begin{ej}
\normalfont
En $D^{\#}=\mathbb{Z}^{\#}$ considérense las siguientes relaciones,
como morfismos en $\text{\textbf{Rel}}(\mathbb{Z}^{\#})$ tales que
$\tau_{1}:\left\{ 2,3\right\} \longrightarrow\mathbb{Z}^{+}$, $\tau_{2}:\mathbb{Z}^{\#}\longrightarrow\mathbb{Z}^{\#}$
y definidos como:
\begin{eqnarray}
\tau_{1} & = & \left\{ (2,2),(2,3),(3,5),(3,7)\right\} \nonumber \\
\tau_{2} & = & \left\{ (n,2n)\,:\,n\in\mathbb{Z}^{\#}\right\} 
\end{eqnarray}
Luego se tiene para $\tau_{1}$ que $Dom\tau_{1}=\left\{ 2,3\right\} $,
$Codom\tau_{1}=\mathbb{Z}^{+}$, $Im\tau_{1}=\left\{ 2,3,5,7\right\} \subset\mathbb{Z}^{+}$
y $Coim\tau_{1}=\left\{ 2,3\right\} $. Además $\tau_{1}$ es una
correspondencia, no es una función parcial puesto que $(2,2)$ y $(2,3)\in\tau_{1}$,
entonces $\text{card}\tau_{1}\left[2\right]=2$, no es sobreyectiva
puesto que $\left\{ 2,3,5,6\right\} =Im\tau_{1}\neq Codom\tau_{1}=\mathbb{Z}^{\#}$
y es inyectiva puesto que $\text{card}\tau_{1}^{-1}\left[x\right]=1$
para todo $x\in Codom\tau_{1}$. Similarmente $Dom\tau_{2}=\mathbb{Z}^{\#}$,
$Codom\tau_{2}=\mathbb{Z}^{\#}$, $Im\tau_{2}=\mathbb{Z}^{\#}\backslash\left\{ 2\right\} $
y $Coim\tau_{2}=\mathbb{Z}^{\#}$ . Se tiene entonces que $\tau_{2}$
es una correspondencia, es función parcial puesto que para todo entero
$x\in\mathbb{Z}^{\#}$ se tiene que $\text{card}\tau_{2}\left[x\right]=1$
cuando $x$ es par y $\text{card}\tau_{2}\left[x\right]=0$ cuando
$x$ es impar, no es sobreyectiva pues que $\left\{ 2n\,:\,n\in\mathbb{Z}^{\#}\right\} =Im\tau_{2}\neq Codom\tau_{2}=\mathbb{Z}^{\#}$
y es inyectiva puesto que para todo $y\in Codom\tau$ se tiene que
$\text{card}\tau^{-1}\left[x\right]=1$ cuando $x$ es par(le corresponde
su mitad) y $\text{card}\tau^{-1}\left[x\right]=0$ cuando $x$ es
impar. En base a los conceptos definidos se obtiene el siguiente resultado:
\end{ej}
\begin{prop}
\label{proposition:section n retract}Dada
una relación $\tau$ en $\text{\textbf{Rel}}(D^{\#})$.

\begin{enumerate}
\item $\tau^{-1}\circ\tau=id_{Dom\tau}$ si y solo si $\tau$ es una correspondencia
inyectiva.
\item $\tau\circ\tau^{-1}=id_{Codom\tau}$ si y solo si $\tau$ es una función
parcial sobreyectiva.
\end{enumerate}
\end{prop}
\begin{proof}
(1) $(\Longrightarrow)$ Por definición se tiene que $Coim\tau\subset Dom\tau$,
falta ver la otra contenencia. Para ello, sea $y\in Dom\tau$, como
$\tau^{-1}\circ\tau=id_{Dom\tau}$ se tiene que $yid_{Dom\tau}y$
si y solo si $y\tau^{-1}\circ\tau y$, así que existe $x\in D^{\#}$
tal que $y\tau x$ y $x\tau^{-1}y$, luego $y\in Im\tau^{-1}=Coim\tau$,
por tanto $\tau$ es una correspondencia. Ahora supóngase que $\tau^{-1}\left[x\right]=\left\{ y_{1},y_{2}\right\} $
con $y_{1},\,y_{2}\in Dom\tau$ y $y_{1}\neq y_{2}$, entonces $y_{1}\tau x$
y $y_{2}\tau x$ luego $x\tau^{-1}y_{2}$ y entonces $y_{1}\tau^{-1}\circ\tau y_{2}$,
pero como $\tau^{-1}\circ\tau=id_{Dom\tau}$ se tiene que $y_{1}id_{Dom\tau}y_{2}$
y por tanto $y_{1}=y_{2}$ y así $\text{card}\tau^{-1}\left[x\right]\leq1$.\\
$\left(\Longleftarrow\right)$ Como $Dom\tau=Coim\tau$, $id_{Dom\tau}\subset\tau^{-1}\circ\tau$
pues si $x\in Dom\tau$ y existe $y\in Codom\tau$ tal que $x\tau y$,
así $y\tau^{-1}x$ y luego $x\tau^{-1}\circ\tau x$. Por otro lado,
dados $x,\,y\in Dom\tau^{-1}\circ\tau$ tales que $x\tau^{-1}\circ\tau y$
entonces existe $z\in Codom\tau$ tal que $x\tau z$ y $z\tau^{-1}y$,
así $z\tau^{-1}x$, pero como $\tau$ es inyectiva, $\text{card}\tau^{-1}\left[z\right]\leq1$
para todo $z\in Codom\tau$, por tanto $x=y$ y entonces $x\tau^{-1}\circ\tau x$
y luego $\tau^{-1}\circ\tau\subset id_{Dom\tau}$.\\
(2) La demostración es análoga a la parte (1), intercambiando los
roles de $\tau$ y $\tau^{-1}$.
\end{proof}
Obsérvese que en la proposición, el hecho de que $\tau^{-1}\circ\tau=id_{Dom\tau}$
implica que $\tau$ es una sección y análogamente $\tau\circ\tau^{-1}=id_{Codom\tau}$
implica que $\tau$ es retracción. Con esto entonces se ha obtenido
una parte para caracterizar los monomorfismos y epimorfismos en $\text{\textbf{Rel}}(D^{\#})$
puesto que toda sección es un monomorfismo y toda retracción es un
epimorfismo.

\section{Resultados de este estudio}

Para completar el proceso, primero se define la siguiente función:
\begin{defi}
\normalfont
Dada una relación $X\overset{\tau}{\longrightarrow}Y$ en $\text{\textbf{Rel}}(D^{\#})$
y $A\subseteq X$, se define la imagen de $A$ bajo $\tau$ como
\[
Im_{\tau}(A)=\left\{ b\in Y\,:\,\exists a\in A,\,a\tau b\right\} .
\]
Si $\mathcal{P}(X)$ denota el conjunto potencia de $X$ y dado $A\in\mathcal{P}(X)$
, se define la función asociada a $\tau$ como, 
\begin{eqnarray}
f_{\tau}:\mathcal{P}(X) & \longrightarrow & \mathcal{P}(Y) \nonumber \\
A & \longmapsto & Im_{\tau}(A).
\end{eqnarray}
Se puede observar que $f_{\tau}$ está bien definida puesto que si
$X\overset{\tau}{\longrightarrow}Y$ es una relación en $\text{\textbf{Rel}}(D^{\#})$,
$A,B\in\mathcal{P}(X)$ son tales que $A=B$ entonces $Im_{\tau}(A)=Im_{\tau}(B)$.
La siguiente proposición relaciona estos nuevos conceptos con los
anteriores.
\end{defi}

\begin{prop}
\label{proposition:1-1, sobre}Dada una relación $X\overset{\tau}{\longrightarrow}Y$
en $\text{\textbf{Rel}}(D^{\#})$. Si $f_{\tau}$ es inyectiva, entonces
$\tau$ es una correspondencia. Análogamente, si $f_{\tau}$ es sobreyectiva,
entonces $\tau$ es sobreyectiva.
\end{prop}

\begin{proof}
Si se asume que $\tau$ no es una correspondencia, entonces existe
$\emptyset\neq A\in\mathcal{P}(X)$ tal que $Im_{\text{\ensuremath{\tau}}}(A)=\emptyset$
. Pero $Im_{\tau}(\emptyset)=\emptyset$ y entonces $f_{\tau}$ no
es inyectiva. Análogamente, asumir que $\tau$ no es sobreyectiva
implica que existe $\emptyset\neq B\in\mathcal{P}(Y)$ tal que $Coim_{\tau}(B)=\emptyset$,
pero $Coim_{\tau}(\emptyset)=\emptyset$ y luego $f_{\tau}(\emptyset)=f_{\tau}(B)$
y $f_{\tau}$ no es función, esto es una contradicción.
\end{proof}
Obsérvese el comportamiento simétrico en la prueba, tal comportamiento
es usual cuando se trabaja con monomorfismos y epimorfismos y se hace
uso de ello en los siguientes resultados. 
\begin{prop}
Una relación $X\overset{\tau}{\longrightarrow}Y$ en $\text{\textbf{Rel}}(D^{\#})$
es un monomorfismo si y solo si $f_{\tau}$ es inyectiva. Análogamente,
$\tau$ es epimorfismo si y solo si $f_{\tau}$ es sobreyectiva.
\end{prop}

\begin{proof}
Se presenta el caso para monomorfismos, el otro caso es análogo.\\
$\left(\Longrightarrow\right)$ Sean $A,\,B\subseteq X$ tales que
$Im_{\tau}(A)=Im_{\tau}(B)$ y $A\neq B$. Sin pérdida de generalidad,
supóngase que existe $x\in A$ tal que $x\notin B$. Luego existe
$y\in Im_{\tau}(A)$ tal que $x\tau y$ y como $Im_{\tau}(A)=Im_{\tau}(B)$
entonces existe $x'\in B$ tal que $x'\tau y$. Sean $X\overset{C_{x}}{\longrightarrow}X$
y $X\overset{C_{x'}}{\longrightarrow}X$ morfismos en $\text{\textbf{Rel}}(D^{\#})$
tales que para todo $a\in X$, $aC_{x}x$ y $aC_{x'}x'$, es decir,
morfismos constantes en $x$ y $x'$ respectivamente. 
\[
\xymatrix{X\ar@<-0.4ex>[r]^{C_{x}}\ar@<0.4ex>[r]_{C_{x'}} & X\ar[r]^{\tau} & Y}
\]
Entonces por la definición de composición se tiene que para todo $a\in X$,
$a\tau\circ C_{x}y$ y $a\tau\circ C_{x'}y$ luego $\tau\circ C_{x}=\tau\circ C_{x'}$,
como $\tau$ es monomorfismo entonces $C_{x}=C_{x'}$ y por tanto
$x=x'$. Esto contradice que $x\notin B$ y por tanto $A=B$, es decir
que $f_{\tau}$ es una función inyectiva.\\
$\left(\Longleftarrow\right)$ Sean $W\overset{\tau_{1}}{\longrightarrow}X$
y $W\overset{\tau_{2}}{\longrightarrow}X$ morfismos en $\text{\textbf{Rel}}(D^{\#})$
tales que $\tau\circ\tau_{1}=\tau\circ\tau_{2}$. 
\[
\xymatrix{W\ar@<-0.4ex>[r]^{\tau_{1}}\ar@<0.4ex>[r]_{\tau_{2}} & X\ar[r]^{\tau} & Y}
\]
Si $a\tau_{1}b$ para algún $b\in X$ entonces por la Proposición
\ref{proposition:1-1, sobre}, como $\tau$ es una correspondencia, existe
$x\in Y$ tal que $b\tau x$. Luego $a\tau\circ\tau_{1}x$ y entonces
$a\tau\circ\tau_{2}x$, entonces existe $c\in X$ tal que $a\tau_{2}c$
y $c\tau x$, esto brinda la existencia de tal composición. Se puede
observar que si dado $A\subseteq W$ tal que $Im_{\tau\circ\tau_{1}}(A)=Im_{\tau\circ\tau_{2}}(A)$
entonces $Im_{\tau}(Im_{\tau_{1}}(A))=Im_{\tau}(Im_{\tau_{2}}(A))$,
esta igualdad se da porque 
\begin{eqnarray}
Im_{\tau}\left(Im_{\tau_{1}}(A)\right) & = & \left\{ x\in Y\,:\,\exists b\in Im_{\tau_{1}}(A),\,b\tau x\right\} \nonumber \\
 & = & \left\{ x\in Y\,:\,\exists b\in X,\,\exists a\in A\,:\,a\tau_{1}b\text{ y }b\tau x\right\} 
\end{eqnarray}
 y por otro lado se tiene que 
\begin{eqnarray}
Im_{\tau\circ\tau_{1}}(A) & = & \left\{ x\in Y\,:\,\exists a\in A,\,a\tau\circ\tau_{1}x\right\} ,\nonumber \\
 & = & \left\{ x\in Y\,:\,\exists a\in A,\,\exists b\in X\,:\,a\tau_{1}b\text{ y }b\tau x\right\} 
\end{eqnarray}
como $f_{\tau}$ es inyectiva se tiene que $Im_{\tau_{1}}(A)=Im_{\tau_{2}}(A)$,
para todo $A\subseteq W$, por lo tanto $\tau_{1}=\tau_{2}$. Así,
$\tau$ es un monomorfimo. 
\end{proof}
Las proposiciones anteriores se pueden resumir en el
siguiente diagrama y análogamente para el caso en el que $\tau$ es
epimorfismo.
\begin{figure}[H]
\[
\xymatrix{\tau\text{ monorfismo} & \tau\text{ sección}\ar@{=>}[l]\ar@{<=>}[r] & \tau\text{ correspondencia inyectiva}\\
 & f_{\tau}\text{ inyectiva}\ar@{<=>}[ul]\ar@{=>}[r] & \tau\text{ correspondencia}
}
\]

\caption{\label{fig:monomorfismos}Implicaciones sobre monomorfismos en $\text{\textbf{Rel}}(D^{\#})$}
\end{figure}

Con los resultados anteriores ya se puede dar una caracterización
completa de los isomorfismos en $\text{\textbf{Rel}}(D^{\#})$, necesaria
para dar las propiedades de una estructura de factorización.
\begin{prop}
Una relación $X\overset{\tau}{\longrightarrow}Y$ en $\text{\textbf{Rel}}(D^{\#})$
es un isomorfismo si y solo si $\tau$ es una función biyectiva. 
\end{prop}

\begin{proof}
$\left(\Longrightarrow\right)$Si $X\overset{\tau}{\longrightarrow}Y$
es un isomorfismo entonces es sección y retracción. Por la proposición
\ref
{proposition:section n retract} se tiene que $\tau$ es una función
biyectiva.\\
$\left(\Longleftarrow\right)$Si $\tau$ es una función biyectiva,
entonces existe la función inversa $\tau^{-1}$ que es tal que $\tau\circ\tau^{-1}=id_{Y}$
y $\tau^{-1}\circ\tau=id_{X}$. Por lo tanto $\tau$ es un isomorfismo.
\end{proof}
Obsérvese que también se puede ahora determinar cuáles son los epimorfimos
extremales en \textbf{Rel}$(D^{\#})$, dado un epimorfismo extremal
$X\overset{e}{\longrightarrow}E$ que se factoriza como $e=m\circ f$,
con $m$ monomorfismo, entonces $m$ es isomorfismo, en \textbf{Rel$(D^{\#})$}
la definición entonces se traduce a: 

Dada una relación $\tau$ en $D^{\#}$, $\tau$ es epimorfismo extremal
si cuando $\tau=\tau_{1}\circ\tau_{2}$ con $\tau_{1}$ una correspondencia
inyectiva, se tiene que $\tau_{1}$ es una relación para la que $f_{\tau}$
es una función biyectiva. Se utiliza ahora la Proposición \ref{proposition:4.9}
para brindar propiedades que debe tener una estructura de factorización
para la categoría \textbf{Rel}$(D^{\#})$. Una $(\mathbf{E},\mathcal{M})$-estructura
de factorización en \textbf{Rel$(D^{\#})$} tiene que tener necesariamente
las siguientes propiedades:
\begin{enumerate}
\item La familia $\mathcal{M}$ está formada por relaciones $\tau$ tales
que $f_{\tau}$ es inyectiva.
\item La familia $\mathcal{M}$ contiene todas las relaciones $\tau$ que
son funciones biyectivas y es cerrada bajo composiciones.
\item La familia $\mathbf{E}$ es cerrada bajo composiciones y contiene
todos los epi-sumideros extremales(vistos como epi-sumideros singuletes),
luego $\mathbf{E}$ contiene los epimorfismos extremales
que se describieron anteriormente(ver párrafo anterior).
\item Las $(\mathbf{E},\mathcal{M})$-factorizaciones son esencialmente
únicas.
\item $\mathcal{M\cap\mathbf{E}}$ contiene todas las relaciones $\tau$
que son funciones biyectivas.
\item Si $\tau_{1}\circ\tau_{2}\in\mathcal{M}$ y $\tau_{2}\in\mathcal{M}$
entonces $\tau_{1}\in\mathcal{M}$. Además si $\tau=\tau_{1}\circ\tau_{2}$
es la $(\mathbf{E},\mathcal{M})$-factorización de $\tau\in\mathcal{M}$,
entonces $\tau_{1}$ debe ser una función biyectiva.
\item Si $\tau_{1}\circ\tau_{2}\in\mathbf{E}$ y $\tau_{2}\in\mathbf{E}$
entonces $\tau_{1}\in\mathbf{E}$.
\end{enumerate}

\section{Conclusiones}

Se han obtenido condiciones para que una relación
$\tau$ en $D^{\#}$ tenga una factorización única de la forma $\tau_{1}\circ\tau_{2}$.
Esto contribuye al estudio de \shortciteA{mendez} en
el sentido que se puede pensar específicamente en qué tipo de composiciones
$\tau_{1}\circ\tau_{2}$ son válidos los resultados allí obtenidos.
Como trabajos futuros, es deseable obtener concretamente una familia
de morfismos que tenga éstas propiedades para casos particulares de
interés. Además, se pretende analizar si existe alguna interacción
entre las condiciones brindadas y los tipos de relaciones que interesan
en la teoría de $\tau$-factorizaciones.{\pagebreak{}}


\bibliographystyle{apacite} 
\renewcommand{\refname}{Referencias}
\bibliography{YOUR_references}

\end{document}